\documentclass[12pt]{amsart}

\usepackage[all]{xy}
\usepackage{latexsym}
\usepackage{rotating}
\usepackage{amsfonts}
\usepackage{amssymb}
\usepackage{amsmath}
\usepackage{hyperref}

\newtheorem{theorem}{Theorem}[section] 
\newtheorem{lemma}[theorem]{Lemma}     
\newtheorem{corollary}[theorem]{Corollary}
\newtheorem{proposition}[theorem]{Proposition}
\newtheorem{remark}[theorem]{Remark}

\newtheorem{conjecture}[theorem]{Conjecture}

\newcommand{\fX}{\mathfrak{X}}

\newcommand{\C}{\mathbb{C}}
\newcommand{\hm}{\mathrm{Hom}}

\newcommand{\quot}{/\!\!/}
\newcommand{\bG}{\mathbf{G}}
\newcommand{\R}{\mathbb{R}}
\newcommand{\Z}{\mathbb{Z}}
\newcommand{\bbR}{\mathbb{R}}
\newcommand{\bbZ}{\mathbb{Z}}
\newcommand{\bbF}{\mathbb{F}}
\newcommand{\bbC}{\mathbb{C}}

\newcommand{\codim}{\mathrm{codim}}
\def\co{\colon\thinspace}
\newcommand{\wt}[1]{\widetilde{#1}}
\newcommand{\cross}{\times}
\newcommand{\homeo}{\cong}
\newcommand{\injects}{\hookrightarrow}
\newcommand{\srm}[1]{\stackrel{#1}{\maps}}
\newcommand{\maps}{\longrightarrow}

\title[Fundamental Groups of Character Varieties]{Fundamental Groups of Character 
Varieties: Surfaces and Tori}

\date{\today}

\author[I. Biswas]{Indranil Biswas}

\address{School of Mathematics, Tata Institute of Fundamental
Research, Homi Bhabha Road, Bombay 400005, India}

\email{indranil@math.tifr.res.in}

\author[S. Lawton]{Sean Lawton}

\address{Department of Mathematical Sciences, George Mason University, 4400 
University Drive, Fairfax, Virginia 22030, USA}

\email{slawton3@gmu.edu}

\author[D. Ramras]{Daniel Ramras}

\address{Department of Mathematical Sciences, Indiana University-Purdue University 
Indianapolis, 402 N. Blackford, LD 270, Indianapolis, IN 46202, USA}

\email{dramras@math.iupui.edu}


 
\thanks{The first named author is supported by the J. C. Bose Fellowship. The second 
named author was partially supported by the Simons Foundation, USA (\#245642) and the 
National Science Foundation, USA (\#1309376).  The third named author was partially 
supported by the Simons Foundation, USA (\#279007).}

\begin{document}

\begin{abstract}
We compute the fundamental group of moduli spaces of Lie group valued representations 
of surface and torus groups.
\end{abstract}

\maketitle

\section{Introduction}

Let $\bG$ be the $\C$-points of a reductive affine algebraic group defined over $\R$ ({\it reductive $\C$-group} for 
short).  We will say a Zariski dense subgroup $G\subset \bG$ is {\it real reductive} if
$\bG(\R)_0\subset G\subset \bG(\R)$, where $\bG(\R)_0$ denotes the connected component of 
the real points of $\bG$ containing the identity element. Let $\Gamma$ be a finitely 
presented discrete group.  Then $G$ acts on the analytic
variety $\hm(\Gamma, G)$ by conjugation (via the adjoint action of $G$ on itself).  
Let $\hm(\Gamma, G)^*\,\subset\,\hm(\Gamma, G)$ be the subspace with closed orbits. Then
the quotient space $\fX_{\Gamma}(G)\,:=\,\hm(\Gamma,G)^*/G$ for the adjoint action is 
called the $G$-character variety of $\Gamma$ (or the moduli space of $G$-representations of
$\Gamma$). By \cite{RiSl}, $\fX_{\Gamma}(G)$ is a Hausdorff space and when $G$ is
real algebraic it is moreover semi-algebraic. Therefore, $\fX_{\Gamma}(G)$
deformation retracts to a compact subspace, which in turn implies
that the fundamental group of $\fX_{\Gamma}(G)$ is finitely generated.  

The spaces $\fX_\Gamma(G)$ constitute a large class of affine varieties 
(see \cite{KaMi, Rap}), and are of central interest in the study of moduli spaces (see \cite{A-B, GGM, Hitchin, NS, Simpson1, Simpson2}).  Additionally they naturally arise in mathematical physics (see \cite{Borel-Friedman-Morgan,Kac-Smilga,Witten}), the study of geometric structures (see \cite{CG2, GoldmanGStructure, Thurston}), knot theory (see \cite{CCGLS}), and even in Geometric Langlands (see \cite{HaTh,KW}).  

Extensive attention over the years has been given to the study of the 
connected components of these 
spaces (see \cite{Goldman-components, Ho-Liu-ctd-comp-I, Ho-Liu-ctd-comp-II, 
Lawton-Ramras}), and recently the systematic study of their fundamental group has begun (see 
\cite{BiLa, Lawton-Ramras}).  In this paper, we further these results to 
the important cases of (orientable) surface groups and free Abelian (torus) groups.

Let $\fX^0_\Gamma (G)$ denote the connected component of the trivial representation in $\fX_\Gamma (G)$.  Here is our main theorem:

\begin{theorem}\label{thm:main}
Let $G$ be either a connected reductive $\C$-group, or a connected compact Lie group, and
let $\Gamma$ be one the following:
\begin{enumerate}
\item a free group,

\item a free Abelian group, or

\item the fundamental group of a closed orientable surface.
\end{enumerate}  
Then 
$\pi_1(\fX^0_{\Gamma}(G))\,=\,\pi_1(G/[G,G])^{r},$ where 
$r\,=\,\mathrm{Rank}\left(\Gamma/[\Gamma,\Gamma]\right)$.
\end{theorem}

From \cite{CFLO} and \cite{FlLaAbelian} we know that $\fX_\Gamma(G)$ and $\fX_\Gamma(K)$ 
are homotopy equivalent whenever $\Gamma$ is free (non-Abelian), or free Abelian, with $G$ being 
a real reductive Lie group with maximal compact subgroup $K$. This homotopy equivalence restricts to a homotopy equivalence $\fX^0_\Gamma (G) \simeq \fX^0_\Gamma (K)$, so from Theorem \ref{thm:main}, 
we then conclude:

\begin{corollary}
When $G$ is real reductive and $\Gamma$ is free non-Abelian or free Abelian, 
then $\pi_1(\fX^0_{\Gamma}(G))\,=\,\pi_1(K/[K,K])^{r},$ where 
$r\,=\,\mathrm{Rank}\left(\Gamma/[\Gamma,\Gamma]\right)$ and $K$ is a maximal compact 
subgroup of $G$.
\end{corollary}

\begin{remark}
{\rm Many of the cases considered in the above theorem, when $\pi_1(G)$ is torsion-free, 
were handled in \cite{Lawton-Ramras} using general results on covering spaces of character varieties. The case
when $\Gamma$ is free (non-Abelian) was first handled in general in \cite{BiLa}, although
with a different proof.}
\end{remark}

\begin{remark}
{\rm For any connected real reductive Lie group $G$ $($which includes the complex reductive 
and compact cases$)$, $G/[G,G]$ is homotopy equivalent to a geometric torus.  Therefore, in the cases 
considered in this paper, $\pi_1(\fX_{\Gamma}(G))\,\cong\, \Z^N$ for some $N$.}
\end{remark}

\begin{corollary}\label{cor1}
Take $G$ and $\Gamma$ as in Theorem \ref{thm:main}, and impose the extra condition that
$G$ is semisimple. Then $\fX^0_{\Gamma}(G)$ is simply connected.
\end{corollary}

\section*{Acknowledgments}
We thank Constantin Teleman and William Goldman for helpful correspondence.

\section{Proof of Main Theorem}

Unless otherwise stated, $G$ is a connected real reductive Lie group. By definition, 
connectedness of $G$ implies that $G$ is the identity component of the real points of 
a reductive $\C$-group.

A continuous map $f\co X\,\to\, Y$ between topological spaces will be called 
$\pi_1$-\textit{surjective} if for each $y\in Y$, there exists $x\in f^{-1} (y)$ such that $f_* \co \pi_1 (X, x) \,\to\, \pi_1 (Y, y)$ is surjective. The following 
lemma appears in \cite{Lawton-Ramras} when $G$ is a reductive $\C$-group.  
However the proof in \cite{Lawton-Ramras} establishes the stronger 
statement below by using results in \cite{RiSl}.  We say that a continuous map $p\co Y\,\to\, Z$ has the (strong) \emph{path-lifting property} if for each path $\alpha\co [0,1] \,\to\, Z$ and each $y\in Y$ with $p(y) = \alpha (0)$, there exists a path $\widetilde{\alpha} \co  [0,1] \,\to\, Y$ such that $\widetilde{\alpha} (0) = y$ and $p\circ \widetilde{\alpha} = \alpha$.  We say that $p$ has the \emph{weak path-lifting property} if for all such $\alpha$ there exists $\widetilde{\alpha}$ such that $p\circ \widetilde{\alpha} = \alpha$.

\begin{lemma}\label{pathlifting}
Let $X$ be a real algebraic subset of an affine space $V$, and let $G$ be a real
reductive Lie group acting linearly on $V$.  Then the projection map $q\co X \to X\quot
G$ has the weak path-lifting property.  Consequently, $q$ is $\pi_1$--surjective when $G$ is connected. 
\end{lemma}

\begin{proof}
By definition, $X\quot G$ is the quotient space $X^*/G$ where $X^*$ is the subspace of points in $X$ with closed $G$-orbits.  By the generalization of Kempf-Ness Theory (see
\cite{KN}) developed in \cite{RiSl}, there exists a real algebraic subset 
$\mathcal{KN}\,\subset\, X$ 
such that $X\quot G$ is homeomorphic to $\mathcal{KN}/K$, where $K$ is a maximal
compact subgroup of $G$. Moreover, the natural diagram
$$
\xymatrix{\mathcal{KN} 
\ar@{->>}[rr] \ar@{^{(}->}[d] & & \mathcal{KN}/K\ar[d]^{\cong} \\ X^* \ar@{^{(}->}[r]^i & X \ar[r]^-q & X\quot G}
$$
commutes. The weak path-lifting property for $q$ now follows from the (weak) path-lifting property for $\mathcal{KN}\to \mathcal{KN}/K$.  Now, since $\mathcal{KN}$ is algebraic, it satisfies the conditions in 
\cite{MY} for there to be a slice at each point. As shown in \cite[page 
91]{Bredon-transf}, this implies that $\mathcal{KN} \, \to \, \mathcal{KN}/K$ has the (strong) path-lifting property.

Now assume $G$ is connected. Given $p\in X\quot G$, choose $x\in X^*$ such that $q(x) = p$.
By the previous paragraph, each loop $\gamma$ based at  $p$  lifts to a path  
$\widetilde{\gamma}$ in $X^*\,\subset\, X$.  However, this path in $X^*$ might not be a loop.  
The ends of the lifted path $\widetilde{\gamma}$ are in the fiber $(qi)^{-1}(p)$, which is 
homeomorphic to $G/\mathrm{Stab}_G(x)$.  Hence the fiber is path-connected
since $G$ is connected. Therefore, there exist paths in $(qi)^{-1}(p)$
that connects the ends of $\widetilde{\gamma}$ to $x$, resulting in a loop in $X^*\,\subset\, X$, based at $x$.  
This loop projects, under $q$, to $\gamma$ (up to reparametrization). Thus $q$ is $\pi_1$--surjective when $G$ is
connected.
\end{proof}

We note that Lemma \ref{pathlifting} includes the cases when $G$ is a reductive 
$\C$-group (so $X\quot G$ is the GIT quotient), and the case when $G$ is 
compact (so $X\quot G$ is the usual orbit space $X/G$).

\begin{corollary}\label{pi1onto}
Let $G$ be a real reductive Lie group that is both connected and algebraic, and let $\Gamma$ be a
finitely presented discrete group.  The natural projection $\hm(\Gamma, G)\,\to\, 
\fX_\Gamma(G)$ is $\pi_1$--surjective, as is the map $\hm^0(\Gamma, G)\,\to\, 
\fX^0_\Gamma(G)$ between the identity components of these spaces.
\end{corollary}

\begin{proof}
Since every reductive $\C$-group has a faithful linear representation, we 
conclude that every real reductive Lie group does as well.  Thus, there is some 
$\mathfrak{gl}(n,\R)$ that contains $G$.  Let $\Gamma$ be generated by $r$ elements.  Then 
the conjugation action of $G$ is linear on the affine space 
$\mathfrak{gl}(n,\R)^{ r}$ which naturally contains $\hm(\Gamma, G)$ via the embedding $\hm(\Gamma, G)\,\subset\, G^{
r}\,\subset\, \mathfrak{gl}(n,\R)^{r}$.  Since $G$ is algebraic and $\Gamma$ is finitely presented, $\hm(\Gamma, G)$ is a real algebraic subset; the result follows by the previous lemma since $G$ is connected. For the statement regarding trivial components, note that the trivial representation has a closed orbit, so we may choose it as our base-point when running the argument in the previous lemma.
\end{proof}

Let $DG\,=\,[G,G]$ be the derived subgroup of $G$.  It is semisimple, normal, 
and closed.  Consequently, $G/DG$ is an Abelian Lie group. Since $G$ is 
connected, so is $G/DG$, and since $G$ is reductive $G/DG$ does not contain an 
affine factor. Therefore, $G/DG$ is a finite product of geometric and algebraic 
tori; that is, a product of copies of $S^1$ and $\C^*$.

Now assume that $G$ is either a reductive $\C$-group or compact (and still
connected). Under this assumption more can be said.  In particular, $G\cong 
(T\times_F DG)$, where $T\,=\, Z_0(G)\, \subset\,Z(G)$ is the connected 
component of the center $Z(G)$ containing the identity element, and $F\,=\,T\cap DG$ is
a central finite subgroup acting freely via $(t,h)\cdot f = (tf, f^{-1} h)$.  Consequently, we have $T/F\,\cong\, 
G/DG$. 

A finitely presented discrete group $\Gamma$ will be called {\it exponent 
canceling} if it admits a finite presentation in which the Abelianization of each relation is trivial (see 
\cite{Lawton-Ramras}).  Consequently, $\Gamma/[\Gamma,\Gamma]$ is free Abelian.  
We will define the rank of $\Gamma$ to be the rank of $\Gamma/[\Gamma,\Gamma]$.
The rank of $\Gamma$ will be denoted by $r$.

If $G$ is Abelian, $\pi_1(\fX_\Gamma(G))=\pi_1(\hm(\Gamma,G))\cong \pi_1(G)^r=\pi_1(G/DG)^r,$ so we assume that $G$ is not Abelian.  

\begin{lemma}\label{fib1}
Let $\Gamma$ be exponent-canceling and $G$ a connected reductive $\C$-group, or a connected compact Lie group.  Then there is a Serre fibration $$\xymatrix{\fX_\Gamma(DG) \ar@{^{(}->}[r] & \fX_\Gamma(G)
\ar@{->>}[r]&\fX_\Gamma(G/DG)}.$$
\end{lemma}

\begin{proof}
Recall that $G\,\cong\, (DG\times T)/F$, where $T$ is a central torus and $F$ 
is a finite central subgroup. 
The multiplication actions of $F$ on $DG$ 
and $T$ commute with the adjoint actions (since $T$ is central) inducing an action of $F^r$ on $\fX_\Gamma (DG) \cross \fX_\Gamma (T)$, and in fact this action is free since it is free on $T^r = \fX_\Gamma (T)$ (this equality uses the fact that $\Gamma/[\Gamma, \Gamma]$ is free Abelian of rank $r$).  Since 
$$F\maps DG\cross T \maps G$$
is a covering sequence of Lie groups, the proof of the main result in 
\cite{Lawton-Ramras}, along with \cite[Proposition 4.2]{Lawton-Ramras}, shows that there is a homeomorphism
$$\fX_\Gamma (G)\,\homeo\, [\fX_\Gamma (DG)\cross  \fX_\Gamma (T)]/F^r\, ,$$
where $F^r \,=\, \hm(\Gamma, F)$ acts diagonally.  Since $F^r$ acts freely on 
$\fX_\Gamma (T)\,=\, T^r$, this action is in fact a covering space action.

Now consider the trivial fibration (product projection)
$$\xymatrix{\fX_\Gamma (DG) \ar@{^{(}->}[r] & \fX_\Gamma (DG)\times \fX_\Gamma (T)
\ar@{->>}[r]&\fX_\Gamma (T)}.$$  
The projection on the right is equivariant with respect to the free actions of $F^r$, so we may apply Corollary \ref{fibration}
to conclude that
 $$\xymatrix{\fX_\Gamma(DG) \ar@{^{(}->}[r] & [\fX_\Gamma (DG)\times \fX_\Gamma (T)]/F^r \homeo   \fX_\Gamma(G) 
\ar@{->>}[r]&\fX_\Gamma(T)/F^r}$$ 
is a Serre fibration.  Since $\fX_\Gamma(T)/F^r\,=\, T^r/F^r\,=\, (T/F)^r\,=\, 
(G/DG)^r\,=\, \fX_\Gamma (G/DG)$, the proof is complete.
\end{proof}

\begin{lemma} \label{fib2} 
If $\Gamma$ is free Abelian, or the fundamental group of an orientable Riemann surface $($closed or open$)$, then the restriction of the fibration from Lemma~\ref{fib1} to the connected component of the trivial representation gives a Serre fibration
 $$\xymatrix{\fX^0_\Gamma(DG) \ar@{^{(}->}[r] &    \fX^0_\Gamma(G) 
\ar@{->>}[r]&\fX_\Gamma(G/DG)} .$$ 
\end{lemma}
\begin{proof}
Since $\fX_\Gamma(G/DG)$ is connected, we just need to show that 
$\fX^0_\Gamma(G)  \cap  \fX_\Gamma(DG) \,= \, \fX^0_\Gamma(DG)$.  

This is obvious in the case of an open surface since in that case $\Gamma$ is a 
free non-Abelian group and so $\fX_\Gamma(G)$ and $\fX_\Gamma(DG)$ are 
connected.

Next, we consider the case in which $\Gamma$ is the fundamental group of a 
closed orientable Riemann surface.  Note that it will suffice to show that the 
natural map $\pi_0 (\fX_\Gamma (DG)) \maps \pi_0 (\fX_\Gamma (G))$ is 
injective.  In the Appendix to \cite{Lawton-Ramras}, Ho and Liu established a 
natural isomorphism $\pi_0 (\hm(\Gamma, H) )\,\homeo\, \pi_1 (DH)$ for 
connected, 
complex reductive groups $H$; they had earlier shown this to be true when $H$ is compact and connected (see \cite{Ho-Liu-ctd-comp-II}).  Note that since $H$ is connected we may replace 
$\hm(\Gamma, H)$ by $\fX_\Gamma (H)$ in this statement.  Since $DG$ is its own derived subgroup, 
the map $\pi_0 (\fX_\Gamma (DG)) \to \pi_0 (\fX_\Gamma (G))$ corresponds to the identity map on $\pi_1 (DG)$ under 
Ho and Liu's isomorphism.

In the free Abelian case, and when $G$ is compact, this can be deduced from the 
natural homeomorphism between $\fX^0_\Gamma (G)$ and $\widetilde{T}^r/W$, where 
$\widetilde{T}\,\subset\,
G$ is a maximal torus and $W$ is its Weyl group.  Note that with respect to the 
decomposition $G\,=\,DG \times_F T$, we have $\widetilde{T}\,=\, T' \times_F T$, where $T'$ 
is a 
maximal torus in $DG$.  Then $\fX^0_\Gamma(DG)\,=\,(T')^r/W$. Clearly then, 
$\fX^0_\Gamma(G)\cap \fX_\Gamma(DG)\,=\, \fX^0_\Gamma(DG)$ given that a 
representation in $\fX^0_\Gamma(G)\cap \fX_\Gamma(DG)$ is in 
$\fX^0_\Gamma(G)\,\cong\,(T' \times_F T)^r/W$ and also has values in $DG$; 
meaning 
the $T$ part must lie in $F$, and so the representation is in 
$(T')^r/W\,\cong\, \fX^0_\Gamma (DG)$.

Now let $G$ be a reductive $\C$-group with maximal compact subgroup $K$, and 
again let $\Gamma$ be free Abelian.  Note that $DK$ is a maximal compact subgroup in 
$DG$ since $G/DG\,\cong\, (\C^*)^d$ and $K/DK\,\cong\,(S^1)^d$ are homotopy equivalent,
where $\dim Z(G)\,=\,d\,=\,\dim Z(K)$. Consider the commutative diagram:  
$$\xymatrix{\fX_\Gamma(DK) \ar@{^{(}->}[r] \ar@{^{(}->}[d]&    \fX_\Gamma(K)\ar@{^{(}->}[d]\\ 
\fX_\Gamma(DG)\ar@{^{(}->}[r]&\fX_\Gamma(G)}$$  The main result in 
\cite{FlLaAbelian} shows that the vertical maps in this diagram are homotopy equivalences, and in 
particular bijections on connected components.  These bijections on components 
clearly send the identity component to the identity component. The argument in 
the previous paragraph shows that the top map does not send any non-identity 
component to the identity component; it follows that the same is true for the 
bottom map.  Therefore, again, we have $\fX^0_\Gamma(G)\cap \fX_\Gamma(DG)\,=\,  
\fX^0_\Gamma(DG)$.
\end{proof}

Thus, we have a long exact sequence in homotopy: $$\cdots\to 
\pi_1(\fX_\Gamma^0(DG))\to \pi_1(\fX_\Gamma^0(G))\to \pi_1(\fX_\Gamma(G/DG))\to\cdots 
$$ However, $\pi_0 (\fX_\Gamma^0(DG))\,=\,0$ since this space is connected.  Also 
since $G/DG$ is Abelian and $\Gamma$ is exponent-canceling of rank $r$, we have
$$\fX_\Gamma(G/DG)\,\cong\, (G/DG)^r.$$ As above, $G/DG$ is a torus, so $\pi_2(\fX_\Gamma(G/DG))\,=\,0$ as 
well.  Hence we find that the above long exact sequence restricts to a short exact 
sequence on fundamental groups.

To complete 
the proof of Theorem \ref{thm:main}, it suffices to prove that $\fX_\Gamma^0(DG)$ is simply connected.
Since $\Gamma$ is exponent canceling of rank $r$, there exists a generating set $\{\gamma_1, \ldots, \gamma_r\}$ for $\Gamma$ in which all relations have trivial Abelianization; we fix one such choice for the remainder of the argument.  We have an associated embedding $\hm(\Gamma, G) \injects G^r$.
Consider the commutative diagram
 $$ \xymatrix{
DG\ar@{^{(}->}[r] \ar[d]& \hm^0(\Gamma,DG)\ar[d]^q \\
DG\quot DG \ar[r] &\fX_\Gamma^0(DG),  \\
} $$
where the top map is the $k$-th factor inclusion $g\,\mapsto\, 
(e,...e,g,e,...e),$ with $1\leq k\leq r$ (which maps into the identity component $\hm^0 (\Gamma, DG)$ of $\hm (\Gamma, DG)$ since $\Gamma$ is 
exponent canceling and $DG$ is connected) and the bottom map is well-defined 
since the $k$-th factor inclusion is $DG$-equivariant (with respect to 
conjugation). By \cite[Corollary 17]{Dal}, for any compact Lie group $H$ we have $\pi_1(H/H)\,=\,\pi_1(H/DH)$; we remind the reader that $H/DH$ is the quotient by left translation and $H/H$ is the quotient by conjugation.  
In particular, if $K\leq G$ is a maximal compact subgroup, then $DK$ is also compact and applying this result with $H= DK$, we find that the conjugation quotient
$DK/DK$ is simply connected.  Now, the main result of \cite{FlLaFree} gives a homotopy equivalence between the conjugation quotients $DG\quot DG$ and $DK/DK$, so we conclude that $DG\quot DG$ is simply connected as well.  It follows that all elements in $\pi_1(DG)$ all map to $0$ in $\pi_1(\fX_\Gamma^0(DG))$ by commutativity.  However, we have shown $q$ is $\pi_1$--surjective (Lemma \ref{pi1onto}).  Together, these observations imply that $\fX_\Gamma^0(DG)$ is simply connected whenever the images of the $k$-th factor inclusions generate $\pi_1(\hm^0(\Gamma, DG))$.  We call this latter property {\it inclusion generating}.

When $\Gamma$ is free non-Abelian, $\hm^0(\Gamma,DG))\,\cong\, (DG)^r$ and thus 
it is obviously inclusion-generating; this gives a shorter proof of the main 
result in \cite{BiLa}, and completes the proof of Theorem \ref{thm:main} when 
$\Gamma$ is free non-Abelian.

In the case when $\Gamma$ is free Abelian, the main results in \cite{GPS} and 
\cite{PeSo} imply that the $k$-th factor inclusions induce an isomorphism
$$\pi_1(DG)^r \,\stackrel{\cong}{\longrightarrow}\, \pi_1(\hm^0(\Gamma,DG))\, .$$
In order to see that the map $\pi_1(DG)^r \to \pi_1(\hm^0(\Gamma,DG))$ is well-defined, 
one needs to know a priori that the images of the $k$-th factor inclusions commute 
with one another.  By the main results in \cite{GPS} and \cite{PeSo}, $\hm^0 (\bbZ^k, 
\wt{DG})$ is simply connected.  Note here that $\wt{DG}$ has the form $\bbF^k \cross 
H_1\cross \cdots \cross H_n$, where $\bbF$ is either $\bbR$ or $\bbC$ and the $H_i$ 
are simply connected, simple Lie groups (see \cite[Section 2]{Lawton-Ramras}, for 
instance).  Now the argument in Lemma 2.6 (below) shows that $ \pi_1(\hm^0(\Gamma,DG))$ is 
Abelian, and hence the $k$-th factor inclusions automatically commute with one 
another. Since $\pi_1(DG)^r \,\stackrel{\cong}{\longrightarrow}\, \pi_1(\hm^0(\Gamma,DG))$,
we again see that $\fX_\Gamma^0(DG))$ is inclusion-generating.
This completes the proof of Theorem \ref{thm:main} when $\Gamma$ is free 
Abelian.

It remains to show that $\hm^0(\Gamma, DG)$ is inclusion generating when $\Gamma$ is a closed hyperbolic surface group (note that the genus 0 case is trivial, and the  genus 1 case is included the free Abelian case above).  This will follow from the next two lemmas.

\begin{lemma}
The space $\hm(\Gamma, G)$ is simply connected when $\Gamma$ is a closed 
hyperbolic 
surface group and $G$ is semisimple, simply connected, and complex, or $G$ is
simply connected and compact.
\end{lemma}

\begin{proof}
By \cite{Li-surface-groups}, we know that $\hm(\Gamma, G)$ is simply connected when $G$ is complex semisimple and simply connected.  We now show this result also holds when $G$ is compact and simply connected.

Let $\Gamma\,=\,\pi_1(M^g)$, where $M^g$ is a closed hyperbolic surface of 
genus $g$ (so $g\geq 2$), and let $G_\C$ be the complexification of $G$.  Let 
$P\, =\, M^g\times G$ be the trivial $G$-bundle over $M^g$ and let $P_\C\,=\,
M^g\times G_\C$ be 
the associated $G_\C$-bundle. The space of all $G$-connections on $P$, 
denoted by $\mathcal{A}$, is naturally identified with the space of all 
holomorphic structures on $P_\C$, denoted by $\mathcal{C}$. Both are infinite 
dimensional complex affine spaces.  Under this isomorphism, the Yang-Mills 
stratification of $\mathcal{A}$, determined by the Yang-Mills functional on $\mathcal{A}$, corresponds to the Harder-Narasimhan stratification of $\mathcal{C}$ (see \cite{A-B, Dask}).

Let $\mathcal{G}$ be the gauge group, which is the group of automorphisms of 
$P$. Let $\mathcal{G}_0\,\subset\, \mathcal{G}$ be the based gauge group,
consisting of all automorphisms that are identity on the fiber $P_{x_0}$ over
the base point $x_0$ of $M^g$. The space $\hm(\Gamma, G)$ is homeomorphic to 
the space of flat connections on 
the trivial $G$-bundle over $M^g$ modulo the action of 
$\mathcal{G}_0$.  In particular, letting $\mathcal{A}_\flat\,\subset\, \mathcal{A}$ 
be the subspace of flat $G$-connections, we have $\hm(\Gamma, G) \,\cong\, 
\mathcal{A}_\flat/\mathcal{G}_0$.

The space $\mathcal{A}_\flat$ is the (unique) minimum critical set of the 
Yang-Mills functional and the stable manifold of this critical set is the space of 
semi-stable holomorphic structures on $P_\C$, denoted by 
$\mathcal{C}_{ss}\,\subset\, \mathcal{C}$.  By the main result of \cite{Rade}, $\mathcal{C}_{ss}$ 
deformation retracts onto $\mathcal{A}_\flat$.

First we show that $\mathcal{A}_\flat$ is simply connected. In Section 10 of 
\cite{A-B}, Atiyah and Bott state a formula for the complex codimension of the 
(disjoint) stratification $\mathcal{C}\,=\,\cup\mathcal{C}_{\mu}$:  $$\codim 
(C_{\mu})\,=\, \sum_{\alpha(\mu) > 0} \left(\alpha (\mu) + g - 1\right)\, .$$  
Here $\mu$ records the Harder-Narasimhan type of a $G_\C$-bundle with 
connection (associated with an element in the positive Weyl chamber), and $C_{\mu}$ is the corresponding stratum.   The symbol $\alpha$ runs over the positive roots of $G$ which implies $\alpha(\mu)\geq 0$.  Since $g\geq 2$, every term in the formula contributes at least $2$ to the complex codimension, excepting only the unique open stratum corresponding to $\mathcal{C}_{ss}$, which we denote by $\mathcal{C}_{\mu_0}$.  Therefore, $\mathcal{C}_{ss}=\mathcal{C}-\cup_{\mu\not=\mu_0}\mathcal{C}_{\mu}$ is simply connected as it is the complement of a countable disjoint union, in a Sobolev space, of locally closed submanifolds of complex codimension at least 2 (see Corollary 4.8 in \cite{Ramras}).  Since $\mathcal{A}_\flat$ is homotopy equivalent to $\mathcal{C}_{ss}$ we conclude that it is simply connected too.

Next we show that $\mathcal{G}_0$ is path connected.  The based gauge group of the trivial $G$-bundle over $M^g$ is the based mapping space $\mathcal{G}_0=\textrm{Map}_* (M^g, G)$.  The homotopy equivalence $G\srm{\simeq} \Omega BG$ (where $\Omega X$ denotes the based loop space)  induces a homotopy equivalence
$$\textrm{Map}_* (M^g, G) \,\simeq \, \textrm{Map}_* (M^g, \Omega BG)$$ and by 
adjointness, $$\textrm{Map}_* (M^g, \Omega BG)\,\cong \,\textrm{Map}_* (\Sigma 
M^g, BG)\, ,$$
where $\Sigma M^g$ is the reduced suspension $S^1\wedge M^g$. Since the 
attaching map for the 2--cell in $M^g$ is a product of commutators, 
its suspension is trivial as an element of $\pi_2 (\Sigma (S^1)^{2g})$; here
we view $(S^1)^{2g}$ as the 1-skeleton of $M^g$.  Hence we have a homotopy 
equivalence $\Sigma M^g\,\simeq\, (\bigvee_{2g} S^2) \vee S^3$, and consequently 
$$\textrm{Map}_* (\Sigma M^g, BG) \,\simeq\, \left(\prod_{2g} \Omega^2 
BG\right) \times \Omega^3 BG \simeq  \left(\prod_{2g} \Omega G\right) \times 
\Omega^2 G\, .$$  Since $\pi_0 ( \Omega G)\,\cong\,\pi_1 (G)\,=\, \{1\}$ and 
$\pi_0(\Omega^2 G)\,\cong\,\pi_2 (G)\,=\,\{1\}$, we find that $\textrm{Map}_* 
(\Sigma M^g, BG)$ is path connected, as required.

Therefore, since we have shown $\mathcal{G}_0$ is connected and 
$\mathcal{A}_\flat$ is simply connected, we conclude that $\hm(\Gamma, G)$ is 
simply connected because the action of $\mathcal{G}_0$ on $\mathcal{A}_\flat$ 
defines a fibration sequence $$\mathcal{G}_0\,\to\, \mathcal{A}_\flat\,\to\, 
\mathcal{A}_\flat/\mathcal{G}_0\cong \hm(\Gamma, G)$$ (see 
\cite{Goldman-Millson}), which in turn gives the sequence 
$$0\,=\,\pi_1(\mathcal{A}_\flat)\,\to\, 
\pi_1(\mathcal{A}_\flat/\mathcal{G}_0)\,\cong\, 
\pi_1(\hm(\Gamma, G))\,\to\, \pi_0(\mathcal{G}_0)\,=\,0\,.$$
This completes the proof of the lemma.\end{proof}

Given that we have now shown that $\hm(\Gamma, G)$ is simply connected when $G$ is 
semisimple and simply connected (complex or compact\footnote{Note that all compact, simply connected Lie groups are semisimple.}), and $\Gamma$ is a hyperbolic 
surface group, the following observation now completes the proof of the closed 
surface group case, and consequently, the proof of the main theorem.  We note that a 
similar argument, in the free Abelian case, appears in \cite[Section 3]{GPS}.

\begin{lemma} Let $G$ be a Lie group with universal cover $\widetilde{G}$, and 
let $\Gamma$ be exponent-canceling.
If $\hm(\Gamma, \widetilde{G})$ is simply connected, then $\hm(\Gamma, G)$ is inclusion-generating.  In fact, the factor inclusions induce an isomorphism
$(\pi_1 (G))^r \,\to\, \pi_1 (\hm^0 (\Gamma, G))$.
\end{lemma}

\begin{proof} Our choice of generators $\gamma_1, \ldots, \gamma_r$ for $\Gamma$ induces a surjection $F_r \,\to\, \Gamma$.  Consider 
the resulting commutative diagram of representation spaces:
$$\xymatrix{ (\pi_1 (G))^r \ar[r]  \ar[d]^=  & \hm (\Gamma, \widetilde{G}) \ar[r] \ar[d]^j & \hm (\Gamma, G) \ar[d]^i\\
			(\pi_1 (G))^r \ar[r]  & \widetilde{G}^r \ar[r] & G^r.}
$$
The bottom row is the $r$--fold product of the universal covering map of $G$, 
hence a normal covering map, and the top row is a normal covering map by 
Goldman~\cite[Lemma 2.2]{Goldman-components} (note that $(\pi_1 (G))^r\,=\, \hm 
(\Gamma, \pi_1 (G))$ since $\Gamma$ is exponent canceling).  We obtain a resulting diagram of long exact sequences in homotopy, which reads in part:
$$\xymatrix{  
			\pi_1 (\hm^0 (\Gamma, \widetilde{G})) = 0 \ar[r]  \ar[d]^=& \pi_1 ( \hm^0 (\Gamma, G)) \ar[r]^-\delta \ar[d]^{i_*} 
				& (\pi_1 (G))^r \ar[r]\ar[d]^= & \pi_0 (\hm^0(\Gamma, \widetilde{G})) = 0 \ar[d]^= \\
				 \pi_1 (\widetilde{G}^r) = 0 \ar[r] & (\pi_1 (G))^r \ar[r]^-\delta  & 
					(\pi_1 (G))^r  \ar[r] & (\pi_0 (\widetilde{G}))^r = 0.}
$$
Note that the boundary maps $\delta$ are homomorphisms since these covering spaces are normal.
It  follows that   $i_*$ is an isomorphism, and the inverse of $i_*$, restricted to any factor of $(\pi_1 (G))^r$, is precisely the corresponding factor-inclusion map.
\end{proof}

We end with a conjecture that is motivated by the results, and the proofs, in this paper:

\begin{conjecture}
Let $G$ be a connected real reductive Lie group, and let $\Gamma$ be an exponent-canceling discrete group.  Then $\pi_1(\fX_{\Gamma}^0(G))\,=\,\pi_1(K/[K,K])^{r},$ where $r=\mathrm{Rank}\left(\Gamma/[\Gamma,\Gamma]\right)$ and $K$ is a maximal compact subgroup of $G$.
\end{conjecture}

\begin{remark}
{\rm The above conjecture cannot be written as 
$\pi_1(\fX_{\Gamma}^0(G))\,=\,\pi_1(G/[G,G])^{r}$ for real reductive Lie groups in
general since, as $G\,=\,\mathrm{U}(p,q)$ demonstrates, $G/DG$ is not always homotopy equivalent to 
$K/DK$ when $G$ is not complex algebraic. On the other hand, as shown in \cite{CFLO, FlLaFree, 
FlLaAbelian}, when $G$ is a real reductive Lie group and $\Gamma$ is free $($Abelian 
or non-Abelian$)$, $\fX_\Gamma(G)$ is homotopy equivalent to $\fX_\Gamma(K)$.}
\end{remark}

\appendix

\section{Serre fibrations}

Recall that a map $f\co E\to B$ has the left lifting property with respect to a map $i\co W\to Z$ if every commutative diagram 
$$\xymatrix{W \ar[r]^-{\wt{H}_0} \ar[d]^-i & E \ar[d]^-f\\ Z \ar[r]^-H & B}$$
admits a lift $Z\to E$ making the diagram commute.
The map $f$ is a \emph{Serre fibration} if it has the left lifting property with respect to the inclusions $[0,1]^{n-1}\cross \{0\} \injects [0,1]^n$ for all $n\geq 1$.  
It is a well-known fact that Serre fibrations satisfy the left lifting property not just for the inclusions $[0,1]^{n-1}\cross \{0\} \injects [0,1]^n$, but also for all inclusions $A\injects B$ where $B$ is a CW-complex, $A$ is a subcomplex, and the inclusion is a homotopy equivalence (this is one part of a model category structure on topological spaces, as constructed in many places).  We need only a very simple special case of this, namely the case of the inclusion $\{\vec{0}\} \injects [0,1]^n$, where $\vec{0} = (0, \ldots, 0)$.  This special case can be proved by a simple induction on $n$: the case $n=1$ is already part of the definition of a Serre fibration, and assuming the result for $n-1$, we factor the inclusion $\{\vec{0}\} \injects [0,1]^n$ through $[0,1]^{n-1}\cross \{0\}$ and apply the left lifting property first to $\{\vec{0}\} \injects [0,1]^{n-1}\cross \{0\}$, and then to $[0,1]^{n-1}\cross \{0\} \injects [0,1]^n$.

\begin{proposition}\label{Serre} Let $X\srm{f} Y\srm{g} Z$ be  maps between topological  spaces, and assume $f$ is surjective.  If  $f$  and $gf$ are Serre fibrations, then so is $g$.
\end{proposition}
\begin{proof} Given a commutative diagram 
$$\xymatrix{ [0,1]^{n-1}\cross \{0\} \ar[r]^-{\wt{H}_0} \ar[d]^-i & Y \ar[d]^-g\\ [0,1]^{n} \ar[r]^-H & Z,}$$
we must produce a map $[0,1]^{n}\to Y$ making the diagram commute.  Since  $f$ is surjective, we may choose a point $x_0\in X$ such that $f(x_0) = \wt{H}_0 (\vec{0})$, and since $X\to Y$ has the left lifting property with respect to $\{\vec{0}\}\injects [0,1]^{n-1}$, there exists a map $G$ making the diagram
$$\xymatrix{ \{\vec{0}\} \ar[r]^-{c_{x_0}} \ar[d] & X \ar[d]^-f\\ [0,1]^{n-1} \ar[ur]^-{G} \ar[r]^-{\wt{H}_0} & Y}$$
commute (where $c_{x_0}  (\vec{0}) = x_0$).
Since $gf$ is a Serre fibration and $g\circ f \circ G = g\circ \widetilde{H}_0 =  H\circ i$, there exists a commutative diagram
\begin{equation*} \xymatrix{ [0,1]^{n-1}\cross \{0\} \ar[r]^-{G} \ar[d]^-i & X \ar[d]^-{g\circ f} \\ [0,1]^{n} \ar[ur]^-{\wt{H}} \ar[r]^-H & Z.}\end{equation*}
The desired lift $[0,1]^n \to Y$ of $H$ is given by $f\circ \widetilde{H}$.
\end{proof}
   
\begin{corollary}\label{fibration}  Let $F$ be a finite group acting freely on    Hausdorff spaces $E$ and $B$, and let $f\co E\to B$ be an equivariant map that is also a Serre fibration.  Then the induced map $E/F\to B/F$ is a Serre fibration.
\end{corollary}   
\begin{proof} We will apply Proposition~\ref{Serre} to the composition
$$E \maps E/F \maps B/F.$$
Quotient maps for free finite group actions on Hausdorff spaces are covering maps, and covering maps are Serre fibrations, so the first map in this sequence is a (surjective) Serre fibration.  The composite map equals the composite map
$$E \maps B \maps B/F,$$
which is a composition of Serre fibrations, hence a Serre fibration.
\end{proof}

\end{document}